\documentclass[11pt]{amsart}
\usepackage{amssymb}
\usepackage{graphics}
\usepackage{latexsym}
\usepackage{amsmath}
\usepackage{amssymb,amsthm,amsfonts}
\usepackage{diagrams}
\usepackage{amscd}
\usepackage[arrow, matrix, curve]{xy}
\usepackage{syntonly}

\title[Global Torelli theorem]{A global Torelli theorem for certain Calabi-Yau threefolds}
\author[Mao Sheng]{Mao Sheng}
\author[Jinxing Xu]{Jinxing Xu}
\email{msheng@ustc.edu.cn}\email{xujx02@ustc.edu.cn}
\address{School of Mathematical Sciences,
University of Science and Technology of China, Hefei, 230026, China}
\begin{document}
%%%%%%%%%%%%%%%%%%%% Text italic %%%%%%%%%%%%%%%%%%%%%%%%%%%%
%\footnote{2010 \textit{Mathematics Subject Classification}. }

\theoremstyle{plain}
\newtheorem{thm}{Theorem}[section]
\newtheorem{theorem}[thm]{Theorem}
\newtheorem*{theorem*}{Theorem}
\newtheorem{lemma}[thm]{Lemma}
\newtheorem{corollary}[thm]{Corollary}
\newtheorem{proposition}[thm]{Proposition}
\newtheorem{addendum}[thm]{Addendum}
\newtheorem{variant}[thm]{Variant}
%%%%%%%%%%%%%%%%%%%% Text roman %%%%%%%%%%%%%%%%%%%%%%%%%%%%%
\theoremstyle{definition}
\newtheorem{construction}[thm]{Construction}
\newtheorem{notations}[thm]{Notations}
\newtheorem{question}[thm]{Question}
\newtheorem{problem}[thm]{Problem}
\newtheorem{remark}[thm]{Remark}
\newtheorem{remarks}[thm]{Remarks}
\newtheorem{definition}[thm]{Definition}
\newtheorem{claim}[thm]{Claim}
\newtheorem{assumption}[thm]{Assumption}
\newtheorem{assumptions}[thm]{Assumptions}
\newtheorem{properties}[thm]{Properties}
\newtheorem{example}[thm]{Example}
\newtheorem{conjecture}[thm]{Conjecture}
\numberwithin{equation}{thm}

% Skriptbuchstaben
\newcommand{\sA}{{\mathcal A}}
\newcommand{\sB}{{\mathcal B}}
\newcommand{\sC}{{\mathcal C}}
\newcommand{\sD}{{\mathcal D}}
\newcommand{\sE}{{\mathcal E}}
\newcommand{\sF}{{\mathcal F}}
\newcommand{\sG}{{\mathcal G}}
\newcommand{\sH}{{\mathcal H}}
\newcommand{\sI}{{\mathcal I}}
\newcommand{\sJ}{{\mathcal J}}
\newcommand{\sK}{{\mathcal K}}
\newcommand{\sL}{{\mathcal L}}
\newcommand{\sM}{{\mathcal M}}
\newcommand{\sN}{{\mathcal N}}
\newcommand{\sO}{{\mathcal O}}
\newcommand{\sP}{{\mathcal P}}
\newcommand{\sQ}{{\mathcal Q}}
\newcommand{\sR}{{\mathcal R}}
\newcommand{\sS}{{\mathcal S}}
\newcommand{\sT}{{\mathcal T}}
\newcommand{\sU}{{\mathcal U}}
\newcommand{\sV}{{\mathcal V}}
\newcommand{\sW}{{\mathcal W}}
\newcommand{\sX}{{\mathcal X}}
\newcommand{\sY}{{\mathcal Y}}
\newcommand{\sZ}{{\mathcal Z}}
% Sonderbuchstaben mit Doppellinie
\newcommand{\A}{{\mathbb A}}
\newcommand{\B}{{\mathbb B}}
\newcommand{\C}{{\mathbb C}}
\newcommand{\D}{{\mathbb D}}
\newcommand{\E}{{\mathbb E}}
\newcommand{\F}{{\mathbb F}}
\newcommand{\G}{{\mathbb G}}
\newcommand{\HH}{{\mathbb H}}
\newcommand{\I}{{\mathbb I}}
\newcommand{\J}{{\mathbb J}}
\renewcommand{\L}{{\mathbb L}}
\newcommand{\M}{{\mathbb M}}
\newcommand{\N}{{\mathbb N}}
\renewcommand{\P}{{\mathbb P}}
\newcommand{\Q}{{\mathbb Q}}
\newcommand{\R}{{\mathbb R}}
\newcommand{\SSS}{{\mathbb S}}
\newcommand{\T}{{\mathbb T}}
\newcommand{\U}{{\mathbb U}}
\newcommand{\V}{{\mathbb V}}
\newcommand{\W}{{\mathbb W}}
\newcommand{\X}{{\mathbb X}}
\newcommand{\Y}{{\mathbb Y}}
\newcommand{\Z}{{\mathbb Z}}
\newcommand{\id}{{\rm id}}
\newcommand{\rank}{{\rm rank}}
\newcommand{\END}{{\mathbb E}{\rm nd}}
\newcommand{\End}{{\rm End}}
\newcommand{\Hom}{{\rm Hom}}
\newcommand{\Hg}{{\rm Hg}}
\newcommand{\tr}{{\rm tr}}
\newcommand{\Sl}{{\rm Sl}}
\newcommand{\Gl}{{\rm Gl}}
\newcommand{\Cor}{{\rm Cor}}
\newcommand{\Aut}{\mathrm{Aut}}
\newcommand{\Sym}{\mathrm{Sym}}
\newcommand{\ModuliCY}{\mathfrak{M}_{CY}}
\newcommand{\HyperCY}{\mathfrak{H}_{CY}}
\newcommand{\ModuliAR}{\mathfrak{M}_{AR}}
\newcommand{\Modulione}{\mathfrak{M}_{1,n+3}}
\newcommand{\Modulin}{\mathfrak{M}_{n,n+3}}
\newcommand{\Gal}{\mathrm{Gal}}
\newcommand{\Spec}{\mathrm{Spec}}
\newcommand{\Jac}{\mathrm{Jac}}

\newcommand{\modulinm}{\mathfrak{M}_{AR}}

\newcommand{\mc}{\mathfrak{M}_{1,6}}

\newcommand{\mx}{\mathfrak{M}_{3,6}}

\newcommand{\tmc}{\widetilde{\mathfrak{M}}_{1,6}}

\newcommand{\tmx}{\widetilde{\mathfrak{M}}_{3,6}}

\newcommand{\smc}{\mathfrak{M}^s_{1,6}}

\newcommand{\smx}{\mathfrak{M}^s_{3,6}}

\newcommand{\tsmc}{\widetilde{\mathfrak{M}}^s_{1,6}}

\newcommand{\tsmx}{\widetilde{\mathfrak{M}}^s_{3,6}}

%%%%%%%%%%%%%%%%%%%%%%%%%%%%%%%%%%%%%%%%%%%%%%%%%%%%%%%%
\newcommand{\proofend}{\hspace*{13cm} $\square$ \\}

\thanks{This work is supported by Chinese Universities Scientific Fund (CUSF), Anhui Initiative in Quantum Information Technologies (AHY150200) and National Natural Science Foundation of China (Grant No. 11622109, No. 11721101).}
\maketitle \centerline{{\itshape To the memory of Yi Zhang}}

\begin{abstract}
We establish a global Torelli theorem for the complete family of Calabi-Yau threefolds arising from cyclic triple covers of $\P^3$ branched along stable hyperplane arrangements.
\end{abstract}

%\tableofcontents

\section{Introduction}
\label{sec:introduction}
Classical Hodge theory attaches complex algebraic varieties with (mixed) Hodge structures and global Torelli theorem asserts that the attached Hodge structures determine complex algebraic varieties up to isomorphism. There are very rare cases of classes of complex algebraic varieties for which global Torelli theorem holds. The renowned examples include the smooth projective curves \cite{Andreotti}, Abelian varieties, polarized K3 surfaces \cite{Piateckii-Shapiro-Safarevi,Burns-Rapoport,Looijenga-Peters} (see also birational global Torelli for hyperk\"{a}hler manifolds \cite{Verbitsky}) and cubic fourfolds \cite{Voisin, Laza, Looijenga}. In this paper, we add one more example into the above list.

\iffalse
The  "Torelli problem" for Calabi-Yau threefolds (CY $3$-folds) asks to what extent a member $X$ of a given deformation family of Calabi-Yau threefolds is determined by the polarized Hodge structure $H^3(X, \ \Z)$. Naturally, one could hope that  the polarized Hodge structure $H^3(X, \ \Z)$ determines the isomorphic class of $X$. The first counterexample to this question was given by Szendr\H{o}i in \cite{Szen}, where he constructed two deformation equivalent and birational CY $3$-folds $X$ and  $Y$,  such that $H^3(X,\ \Z)\simeq H^3(Y,\ \Z)$ as polarized Hodge structures,  but $X$  and $Y$ are not isomorphic. In \cite{BCP} and \cite{OR}, stronger examples of deformation equivalent, non-birational CY $3$-folds with equivalent middle Hodge structures are constructed.

In general, we say a hyperplane arrangement $\{H_1,\cdots, H_6\}$ in $\P^3$ are in general position, if the divisor $\sum_{i=1}^6 H_i$ is simple normal crossing. Given such a hyperplane arrangement $\{H_1,\cdots, H_6\}$, let $X$ be the cyclic triple cover of $\P^3$ branched along the divisor $\sum_{i=1}^6 H_i$. By simple computations of canonical bundles, one sees that any crepant resolution of $X$ is a CY $3$-fold. In \cite{SXZ}, we proved the existence of such  crepant resolutions.

The main result of this paper
\fi

Our example stems from our early studies \cite{SXZ, SXZ2} on Calabi-Yau varieties arising from cylic covers of projective spaces branched along hyperplane arrangements.
They are cyclic triple covers of $\P^3$ branched along six hyperplanes which are stable in the sense of GIT \cite{DO} (for hyperplane arrangements in general position (non-unique) crepant resolutions of singular CY varieties exist \cite{SXZ}). We show that the attached weight 3 Hodge structues of these CY varieties are pure (Propositions \ref{prop:summarize Hodge structures on X, C}, \ref{prop:stable CY has pure Hodge structure}). The main result of the paper is the following:
\begin{theorem*}\label{main thm: introduction}
Let $X$ and $Y$ be two Calabi-Yau threefolds which are cyclic triple covers of $\P^3$ branched along stable hyperplane arrangements. Then $X$ and $Y$ are isomorphic if and only if $H^3(X,\Z)$ and $H^3(Y,\Z)$ are isomorphic as polarized Hodge structure.
\end{theorem*}
The theorem is a direct consequence of Theorems \ref{thm:global Torelli}, \ref{thm:global Torelli:stable case} in the text. The method to establish the results is pretty standard, namely using the variation of Hodge structure attached to a universal family. The analysis of the corresponding period map (and its extension to stable locus) is based on the celebrated work of Deligne-Mostow \cite{DM}.
\iffalse
Our main strategy to the proof of this theorem is to analyze the period map for the universal family $f_3$ of these branched cyclic triple covers of $\P^3$. By results in \cite{SXZ2}, we can relate this family of CY $3$-folds to the family $f$ of cyclic triple covers of $\P^1$ branched along six points. Then we compute the monodromy group of the family $f$ and describe the period map based on Deligne-Mostow's results in \cite{DM}. By the relationship of the period maps of $f_3$ and $f$, we can conclude that the period map of $f_3$ induces an injective map from the moduli space of six hyperplane arrangements in $\P^3$ in general position to the quotient space of the period domain by the monodromy group. Then Theorem \ref{main thm: introduction} follows from this injectivity of period map. An additional degeneration analysis gives the analogous result for stable hyperplane arrangements.
\fi

Inspired by a discussion with Chin-Lung Wang on a possible classification of complete Calabi-Yau families with Yukawa coupling length one, we were led to compare our example with other known families with Yukawa coupling length one. Surprisingly, it turns out our example is essentially \emph{not} new. Indeed, we show in Appendix that our family over the smooth locus formed by hyperplane arrangements in general position is birationally equivalent to the one constructed by J. Rohde \cite{Rohde} via Borcea-Voisin construction (Propositions \ref{prop:Rohde's CY birational to our CY}, \ref{prop:family version:Rohde's CY birational to our CY}). To the authors' best knowledge, our example (which is essentially the one of Rohde) is the \emph{only} example in literature of complete families of CY varieties with the global Torelli property.

\iffalse
Our next result is about the comparison between our family $f_3$ and the family of CY $3$-folds constructed by J. Rohde in \cite{Rohde}. Rohde's CY $3$-folds have the same Hodge numbers with ours. Moreover, the Griffiths-Yukawa coupling length of both families are one. In section \ref{sec:comparison with Rohde's example}, we show  these two families of CY $3$-folds are birationally isomorphic, and hence give an explanation of the coincidence  of Hodge numbers.
\fi

\section{Families from hyperplane arrangements}\label{sec:Families from hyperplane arrangements}
Given a hyperplane arrangement $\mathfrak{A}$ in $\P^n$ in general position, the cyclic cover of $\P^n$ branched along $\mathfrak{A}$ is an interesting algebraic variety. When the hyperplane arrangement $\mathfrak{A}$ moves in the coarse moduli space of hyperplane arrangements, we get a family of projective variety. In this section, we collect some known facts about Hodge structures of these cyclic covers.

We say an ordered arrangement $\mathfrak{A}=(H_1,\cdots, H_m)$ of hyperplanes in $\P^n$ is in general position if no $n+1$ of the hyperplanes intersect in a point, or equivalently, if the divisor $\sum_{i=1}^mH_i$ has simple normal crossings.

Given an odd number $n$, we set $r=\frac{n+3}{2}$. Then for each ordered hyperplane arrangement $(H_1,\cdots, H_{n+3})$ in $\P^n$  in general position, we can define a (unique up to isomorphism) degree $r$ cyclic cover of $\P^n$ branched along the divisor $\sum_{i=1}^{n+3}H_i$. In this way, if we denote  the coarse moduli space of   ordered $n+3$ hyperplane arrangements  in $\P^n$ in general position by $\mathfrak{M}_{n, n+3}$, then we obtain a universal family $f_n:\mathcal{X}_{AR}\rightarrow \mathfrak{M}_{n, n+3}$ of degree $r$ cyclic covers of $\P^n$ branched along $n+3$ hyperplane arrangements in general position. In  \cite{SXZ}, we constructed a simultaneous crepant resolution $\pi: \tilde{\mathcal{X}}_{AR}\rightarrow \mathcal{X}_{AR}$ for the family $f$ without changing the middle cohomology of fibers. Moreover, this simultaneous crepant resolution gives an $n$-dimensional projective  Calabi-Yau family which  is maximal in the sense that its Kodaira-Spencer map is an isomorphism at each point of $\mathfrak{M}_{n, n+3}$. We denote this smooth projective Calabi-Yau family by $\tilde{f}_n:\tilde{\mathcal{X}}_{AR}\rightarrow \mathfrak{M}_{n, n+3}$.

Now we recall the relation between a cyclic cover of $\P^1$ branched along points and that of $\P^n$ branched along hyperplane arrangements.

Suppose $(p_1,\cdots, p_{n+3})$ is a collection of $n+3$ distinct points on $\P^1$, and put $H_i=\{p_i\}\times \P^1\times \cdots \times \P^1$. By the natural identification between $\P^n$ and the symmetric power $Sym^n(\P^1)$ of $\P^1$, we can view each $H_i$ as a hyperplane in $\P^n$.  Then it can be shown that $(H_1,\cdots, H_{n+3})$ is a hyperplane arrangement in $\P^n$ in general position (\cite{SXZ}, Lemma 3.4). A direct computation shows that this construction gives an isomorphism between the moduli space $\mathfrak{M}_{1, n+3}$ and $\mathfrak{M}_{n, n+3}$ (\cite{SXZ}, Lemma 3.5). Moreover, for $r=\frac{n+3}{2}$, if we denote $C$ as the $r$-fold cyclic cover of $\P^1$ branched along the $n+3$ points, and $X$ as the $r$-fold cyclic cover of $\P^n$ branched along the corresponding hyperplane arrangement $(H_1,\cdots, H_{n+3})$, then we have an isomorphism
$$%\begin{equation}\label{equation:X\simeq C^n/N\rtimes S_n }
X \simeq C^n/N\rtimes S_n.
$$%\end{equation}
Here $N$ is the kernel of the summation homomorphsim $(\Z/r\Z)^n \rightarrow \Z/r\Z$. The action of $\Z/r\Z$  on $C$ is induced from the cyclic cover structure, and $S_n$ acts on $C^n$ by permutating the $n$ factors.

We summarize the properties of the Hodge structures on $X$ and $C$ as the following proposition (\cite{SXZ}, Lemma 2.7, Proposition 3.7):
\begin{proposition}\label{prop:summarize Hodge structures on X, C}
Suppose $n$ is an odd number, and $(p_1,\cdots, p_{n+3})$ is a collection of $n+3$ distinct points on $\P^1$. For each $1\leq i\leq n+3$, put $H_i=\gamma(\{p_i\}\times \P^1\times \cdots \times \P^1)$, viewed as a hyperplane in $\P^n$. Let $r=\frac{n+3}{2}$,  and $C$ be the $r$-fold cyclic cover of $\P^1$ branched along $\sum_{i=1}^{n+3}p_i$. Suppose $X$ is the $r$-fold cyclic cover of $\P^n$ branched along $\sum_{i=1}^{n+3}H_i$. Then we have:
\begin{enumerate}
\item[(1)] The natural $\Q$-mixed Hodge structure on the middle cohomology group $H^n(X, \Q)$ is pure.
\item[(2)] $H^n(X,\C)_{(i)}\simeq \wedge^n H^1(C,\C)_{(i)}$, for each $1\leq i \leq r-1$,
\end{enumerate}
\end{proposition}
Here and from now on, we fix a primitive $r$-th root of unity $\zeta$, a generator $\sigma$ of the cyclic group $\Z/r\Z$, and we use $H^n(X,\C)_{(i)}$ to denote the $i$-eigenspace $\{\alpha \in H^n(X, \C)| \sigma \alpha=\zeta^i \alpha\}$ of $H^n(X,\C)$. The notation $H^1(C,\C)_{(i)}$ has the similar meaning.

\begin{remark}\label{remark:f_n and tildef_n}
Since the  simultaneous crepant resolution $\tilde{\mathcal{X}}_{AR}\rightarrow \mathcal{X}_{AR}$ of the universal family $\mathcal{X}_{AR}\xrightarrow{f} \mathfrak{M}_{n,n+3}$ does not change the middle cohomologies of the fibers, the two $\Q$-PVHS (rational polarized variation of  Hodge structures) $R^n\tilde{f}_*\C_{\mathcal{X}_{AR}}$ and $R^n f_*\C_{\mathcal{X}_{AR}}$ are isomorphic.
\end{remark}

\section{The monodromy group and period map: curve case}\label{sec:monodromy group and period map:curve case}
In this section, we first determine %the monodromy group of the family $\tilde{f}_3:\tilde{\mathcal{X}}_{AR}\rightarrow \mathfrak{M}_{3, 6}$ of Calabi-Yau threefolds. By Proposition \ref{prop:summarize Hodge structures on X, C} and Remark \ref{remark:f_n and tildef_n}, we first analyze
the monodromy group of the universal family of cyclic triple covers of $\P^1$ branched along six distinct points. Then we recall Deligne-Mostow' result about period maps of this family.

Take five distinct points $a_1, \cdots, a_5\in \C$. Let $C$ be the smooth projective curve whose affine model is defined by the equation  \begin{displaymath}
\{(x,y)\in \C^2| y^3=\prod_{i=1}^5 (x-a_i)\}.
\end{displaymath}
 The cyclic triple covering  structure induces   a natural automorphism of  $C$:
\begin{equation}\notag
\begin{split}
\sigma: C &\rightarrow C \\
(x,y)&\mapsto (x, \omega y)
\end{split}
\end{equation}
where $\omega=\exp(\frac{2\pi\sqrt{-1}}{3})$ is a primitive cubic root of unity.

%Via this action, we define
%$$
%H^1(C, \ \Z[\omega])_{\omega^i}:=\{\alpha\in H^1(C, \ \Z[\omega])| \sigma^*\alpha = \omega^i \alpha\};\ \  i=1,2.
%$$

We have the decomposition of $\Z[\omega]$-modules:
$$
H^1(C, \ \Z)\otimes_{\Z}\Z[\omega]=H^1(C, \ \Z[\omega])_{\omega}\oplus H^1(C, \ \Z[\omega])_{\bar{\omega}}
$$
where
$$
H^1(C, \ \Z[\omega])_{\omega^i}:=\{\alpha\in H^1(C, \ \Z[\omega])| \sigma^*\alpha = \omega^i \alpha\};\ \  i=1,2.
$$

The formula    $h(\alpha, \beta):=- i  Q(\alpha, \bar{\beta})$ defines an  Hermitian form
$$
h: H^1(C, \ \Z[\omega])_{\omega} \times H^1(C, \ \Z[\omega])_{\omega} \rightarrow \Z[\omega],
$$
where $Q$ means the intersection pairing  on  $H^1(C, \ \Z)$. We can verify that $H^1(C, \ \Z[\omega])_{\omega}$ is a rank four free $\Z[\omega]$-module, and the Hermitian form $h$ is unimodular with signature $(3,1)$.

Now we can describe the monodromy group in the curve case. Let $(\Lambda, h)$ be a fixed $\Z[\omega]$-lattice of signature $(3,1)$, and let $\mc:=\{(z_1,\cdots, z_6)\in (\P^1)^6| z_i\neq z_i, \ \forall i\neq j\}/PGL(2, \ \C)$  be the moduli space of ordered six distinct points on $\P^1$. Let $f: \mathcal{C}\rightarrow \mc$ be the universal family  of  cyclic triple covers of $\P^1$ branched along six distinct points. Fix a base point $s\in \mc$, and let $C:=f^{-1}(s)$ be the fiber over $s$. Then we have the monodromy representation  $\rho: \pi_1(\mc, s)\rightarrow Aut(H^1(C, \ \Z[\omega])_{\omega}, h)$. Since $(H^1(C, \ \Z[\omega])_{\omega}, h)\simeq (\Lambda, h)$, we can view the monodromy group $\Gamma:=\rho(\pi_1(\mc, s))$ as a subgroup of $Aut(\Lambda, h)$.

In order to describe $\Gamma$, we first introduce some notations. Let $\theta:=\omega-\bar{\omega}=\sqrt{3}i$. Then $V:= \Lambda/\theta \Lambda$ is a four dimensional vector space over the finite field $\F_3\simeq \Z[\omega]/\theta \Z[\omega]$, and $h$ reduces to a quadratic form $q$ on $V$. Let $\psi: Aut(\Lambda, h)\rightarrow Aut(V, q)$ be the natural reduction map, and let $v: Aut(V, q)\rightarrow \F_3^*/\F_3^{*2}\simeq \Z/ 2\Z$ be the spinor norm. Define the following notations:
\begin{equation}\notag
\begin{split}
&Aut^{+}(V,q):=ker v \\
&Aut^{+}(\Lambda, h):= ker (v\circ \psi)\\
&\Gamma_{\theta}:=ker \psi\\
&S:=\{\pm 1, \pm \omega, \pm \omega^2\}\subset Aut(\Lambda, h)\\
&S_0:=\{1, \omega, \omega^2\}\subset S\\
&PAut(\Lambda, h):= Aut(\Lambda, h)/S\\
%&PAut^{+}(\Lambda, h):= Aut^{+}(\Lambda, h)/S_0 \\
&P\Gamma_{\theta}:=\Gamma_{\theta}/S_0
\end{split}
\end{equation}

\begin{proposition}\label{prop:monodromy of mc}
$\Gamma$ is a subgroup of $\Gamma_{\theta}$, and $\Gamma_{\theta}$ is generated by $\Gamma$ and $\omega$.
\end{proposition}

As a direct consequence, we have the following corollary.
\begin{corollary}\label{cor:projectified monodromy group}
The projectified monodromy representation
$$
P\rho: \pi_1(\mc, s)\rightarrow PAut(\Lambda, h)
$$
has image $P\Gamma_{\theta}$.
\end{corollary}

Before starting the proof of Proposition \ref{prop:monodromy of mc}, we first introduce some auxiliary spaces and families of curves over them. Define  $M^{'}:=\{(z_1,\cdots, z_6)\in \C^6|\  \forall i\neq j, \ \ z_i\neq z_j.\}$. Obviously the permutation group $S_6$ acts on $M^{'}$. Let $\bar{M^{'}}:=M^{'}/S_6$ be the quotient space. Similarly as the universal family $\mathcal{C}\xrightarrow{f} \mc$, we have a universal family $\mathcal{C}^{'}\rightarrow M^{'}$, such that for each $s=(z_1,\cdots, z_6)\in M^{'}$, the fiber $\mathcal{C}^{'}_s$ is the cyclic triple cover of $\P^1$ branched along the six points $z_1,\cdots, z_6$.  Obviously this family $\mathcal{C}^{'}\rightarrow M^{'}$ descends to a family $\bar{\mathcal{C}}^{'}\rightarrow \bar{M^{'}}$ over $\bar{M^{'}}$. Fixing base points $s\in M^{'}$ and $\bar{s}\in \bar{M^{'}}$, we also have the monodromy representations $\rho:\pi_1(M^{'}, s)\rightarrow Aut(\Lambda, h)$ and $\rho:\pi_1(\bar{M^{'}}, \bar{s})\rightarrow Aut(\Lambda, h)$.

%\begin{lemma}\label{lemma:relations between M and mc}
%We compare the various  monodromy representations
%\end{lemma}

\begin{lemma}\label{lemma:monodromy of barM and M}The monodromy groups can be determined as follows:
\begin{itemize}
\item[(1)]The monodromy group of the family $\bar{\mathcal{C}}^{'}\rightarrow \bar{M^{'}}$ over $\bar{M^{'}}$ is $\rho(\pi_1(\bar{M^{'}}, \bar{s}))=Aut^{+}(\Lambda, h)$.
\item[(2)] The monodromy group of the family $\mathcal{C}^{'}\rightarrow M^{'}$ over $M^{'}$ is $\rho(\pi_1(M^{'}, s))=\Gamma_{\theta}$.
\end{itemize}
\end{lemma}

\begin{proof}
(1) It is well known that the fundamental group $\pi_1(\bar{M^{'}}, \bar{s})$ is isomorphic to the braid group $B_6$. In order to describe it, we suppose the base point $\bar{s}$ represents the unordered subset $B=\{1,2,\cdots, 6\}$ of $\C$. It is a standard fact that $B_6$ is isomorphic to $Mod_c(\C, B)$, the compactly-supported mapping-class group of the pair $(\C, B)$. Moreover, $B_6$ admits standard generators $\tau_1,\cdots, \tau_5$, where $\tau_i$ denotes the right Dehn half-twist around  a loop enclosing  the interval $[i,i+1]$ in $\C$. These five generators satisfy the braid relation
$$
\tau_i\tau_{i+1}\tau_i=\tau_{i+1}\tau_i\tau_{i+1}
$$
for $i=1,\cdots, 5$, as well as the commutation relation $\tau_i\tau_{j}=\tau_j\tau_i$ for $|i-j|>1$; and these generators and relations give a presentation for $B_6$.

As for the monodromy action of $B_6$ on $\Lambda$, we can see $\tau_i$ acts on $\Lambda$ as a $-\omega$-reflection. Then an application of Lemma (7.12) in \cite{ACT} shows that $\rho(\pi_1(\bar{M^{'}}, \bar{s}))=Aut^{+}(\Lambda, h)$.

(2) We know that $\pi_1(M^{'}, s)\simeq PB_6$, the pure braid group of six points, and $PB_6$ is generated by $\tau_1^2, \tau_2^2,\cdots, \tau_5^2$. Moreover, $PB_6$ is normal subgroup of $B_6$, and the quotient group $B_6/PB_6$ is isomorphic to $S_6$, the permutation group of six elements.

By  the arguments in (1), we see $\rho(\tau_i)$ is a $-\omega$-reflection in $\Lambda$, for $i=1,2,\cdots, 5$. So $\rho(\tau_i^2)$ is a $\omega$-reflection in $\Lambda$, for $i=1,2,\cdots, 5$, and hence by the definition of $\Gamma_{\theta}$, the monodromy group $\rho(\pi_1(M^{'}, s))$ is contained in $\Gamma_{\theta}$.  By Lemma (4.5) in \cite{ACT}, we have a short exact sequence of group $1\rightarrow \Gamma_{\theta}\rightarrow Aut^{+}(\Lambda, h)\rightarrow Aut^{+}(V, q)\rightarrow 1 $. Then we get
$$
[Aut^{+}(\Lambda, h): \rho(\pi_1(M^{'}, s))]\geq [Aut^{+}(\Lambda, h): \Gamma_{\theta}]=|Aut^{+}(V, q)|=720.
$$

On the other hand,
$$
[Aut^{+}(\Lambda, h): \rho(\pi_1(M^{'}, s))]=[\rho(\pi_1(\bar{M^{'}}, \bar{s})):\rho(\pi_1(M^{'}, s))]\leq [B_6:PB_6]=|S_6|=720.
$$

So we obtain $\rho(\pi_1(M^{'}, s))=\Gamma_{\theta}$.

\end{proof}

\begin{remark}
As a byproduct of  the proof of Lemma \ref{lemma:monodromy of barM and M}, we get the commutative diagram:
\begin{equation}\label{diagram:some groups}
\begin{diagram}[labelstyle=\scriptscriptstyle]
1       & \rTo &PB_6 &\rTo& B_6 & \rTo & S_6         & \rTo &1 \\
        &      & \dTo&    &\dTo &      & \dTo^{\wr} &      &   \\
 1       & \rTo &\Gamma_{\theta} &\rTo& Aut^{+}(\Lambda, h) & \rTo & Aut^{+}(V, q)        & \rTo &1 \\
\end{diagram}
\end{equation}
where the rows are short exact sequences and the homomorphism $S_6\rightarrow Aut^{+}(V, q)$ is an isomorphism.

\end{remark}

\textbf{Proof of Proposition \ref{prop:monodromy of mc}:}

Let $M^{''}:=\{(z_1,\cdots, z_5)\in \C^5| z_i\neq z_j, \ \ \forall i\neq j\}$ be the moduli space of five distinct ordered points in $\C$. Let $\mathcal{C}^{''}\rightarrow M^{''}$ be the family of smooth projective curves whose affine model is defined by $y^3=\prod_{i=1}^5(x-z_i)$.

We have the following inclusion of  moduli spaces:
\begin{equation}\notag
\begin{diagram}[labelstyle=\scriptscriptstyle]
\mc &\rInto^i &M^{''}
\end{diagram}
\end{equation}
Here we identify $\mc$ as the space $\{(z_1, z_2, z_3)\in (\C\backslash\{0,1\})^3| z_i\neq z_j,  \ \forall i\neq j\}$, and $i$ maps a point $(z_1, z_2, z_3)$ to the point $(0,1, z_1, z_2, z_3)$ in $M^{''}$.

It is easy to see that, through the map $i$, the space $M^{''}$ is homeomorphism to the product space $\mc\times \C \times \C^*$.
Moreover, the restriction $\mathcal{C}^{''}|_{\mc}$ is isomorphic to the family $\mathcal{C}$ over  $\mc$. So the monodromy group of $\mathcal{C}^{''}|_{\mc}\rightarrow \mc$ is also $\Gamma$. On the other hand, it can be seen directly from the construction that, the monodromy groups of the two families $\mathcal{C}^{''}\rightarrow M^{''}$ and $\mathcal{C}^{'}\rightarrow M^{'}$ are isomorphic. We then identify these two monodromy groups, and by Lemma \ref{lemma:monodromy of barM and M} (2), we know the monodromy group of $\mathcal{C}^{''}\rightarrow M^{''}$ is $\Gamma_{\theta}$.

Now we compare the monodromy groups of $\mathcal{C}^{''}\rightarrow M^{''}$ and its restriction $\mathcal{C}^{''}|_{\mc}\rightarrow \mc$.
Since $M^{''}\simeq \mc\times \C \times \C^*$, by fixing a base point $s=(z_1,\cdots, z_5)\in \mc$, the fundamental group $\pi_1(M^{''}, s)$ is generated by $\pi_1(\mc, s)$ and a loop $\mu: \theta \mapsto (e^{i\theta}z_1, \cdots, e^{i\theta}z_5)$, $0\leq \theta\leq 2\pi$. It can be verified directly that the monodromy action induced by $\mu$ is $\rho(\mu)=\omega\in \Gamma_{\theta}$. So we obtain that the monodromy group of the family $\mathcal{C}^{''}\rightarrow M^{''}$ is generated by $\omega$ and the monodromy group of $\mathcal{C}^{''}|_{\mc}\rightarrow \mc$. This in turn implies that   $\Gamma_{\theta}$ is generated by $\Gamma$ and $\omega$.  \proofend

Now we describe the period map of the family $f: \mathcal{C}\rightarrow \mc$. Recall $(\Lambda, h)$ is a $\Z[\omega]$-lattice of signature $(3,1)$. Let $\B_3:=\{v\in \P(\Lambda\otimes_{\Z[\omega]}\C) | h(v, v)<0)\}$, which is isomorphic to  the three dimensional unit ball $\B_3$. For  any point $s\in \mc$, the space $H^{1,0}(\mathcal{C}_s, \C)_{\bar{\omega}}:=\{\alpha\in H^1(C, \ \C)| \sigma^*\alpha = \bar{\omega} \alpha\}$  is a one-dimensional linear space over $\C$. By associating to $s$ the line in $H^1(\mathcal{C}_s, \C)_{\bar{\omega}}$ generated by $H^{1,0}(\mathcal{C}_s, \C)_{\bar{\omega}}$, we get  a well defined holomorphic map (called the period map):
$$
P_C: \mc^{uni}\rightarrow \B_3
$$
where $\mc^{uni}$ is the universal cover of $\mc$.

By definition,  $P_C$ is equivariant under the projectified monodromy representation:
$$
P\rho: \pi_1(\mc, s)\rightarrow    PAut(\Lambda, h).%P\Gamma_{\theta}\subset
$$
Let $K$ be the kernel of $P\rho$, then $P_C$ descends to the following  holomorphic map, still denoted by $P_C$
$$
\tmc:=\mc^{uni}/K\rightarrow \B_3.
$$

Let $\smc:=\{f:\{1,2,\cdots, 6\}\rightarrow (\P^1)^6| \forall a\in \P^1, \sharp f^{-1}(a,a,\cdots, a)\leq 2\}/PGL(2)$ be the moduli space of stable six ordered points on $\P^1$. Then$\smc$ is a smooth complex manifold and $\smc \backslash \mc$ is a normal crossing divisor. If we denote  $\tsmc\rightarrow \smc$ as the Fox completion of $\tmc\rightarrow \mc$, then By \cite{DM}, the period map $P_C:\tmc\rightarrow \B_3$ extends to an isomorphism $\tsmc\xrightarrow{\sim} B_3$. The following proposition is the starting point of our global Torelli theorem.
\begin{proposition}\label{prop:period iso: curve case}
The period mapping induces a bijective map: $\smc/S_6 \xrightarrow{\sim} \B_3/Aut(\Lambda, h)$.
\end{proposition}
\begin{proof}
By Corollary \ref{cor:projectified monodromy group}, the covering $\tmc\rightarrow \mc$ is a Galois cover with deck transformation group $P\Gamma_{\theta}$, so we have
$\smc\simeq \tsmc/P\Gamma_{\theta}\simeq \B_3/P\Gamma_{\theta}$. We have seen from the diagram (\ref{diagram:some groups}) that $PAut(\Lambda, h)/ P\Gamma_{\theta}\simeq Aut^+(V, q)\simeq S_6$, so we get $\smc/S_6 \simeq  \B_3/PAut(\Lambda, h)=\B_3/Aut(\Lambda, h)$.
\end{proof}

\section{The monodromy group and period map: Calabi-Yau threefold case}

In this section, we analyze the monodromy group and period map of the universal family $f_3:\mathcal{X}_{AR}\rightarrow \mathfrak{M}_{3, 6}$, which is the  family of cyclic triple covers of $\P^3$ branched along six hyperplane arrangements in general position. Our strategy is to use the correspondence between this family and the family of curves considered in the previous section.

Let $H_1,\cdots, H_6$ be six hyperplanes in general position in $\P^3$. Let $X$ be the cyclic triple cover  of $\P^3$ branched along the divisor $\sum_{i=1}^6 H_i$. Similarly with the curve case, we have a natural $\Z/3\Z=<\sigma>$ action on $X$, and we have the eigen-subspace decomposition
$$
H^3(X, \ \Z[\omega])=H^3(X, \ \Z[\omega])_{\omega}\oplus H^3(X, \ \Z[\omega])_{\bar{\omega}}
$$
where
$$
H^3(X, \ \Z[\omega])_{\omega^i}:=\{\alpha\in H^3(X, \ \Z[\omega])| \sigma^*\alpha = \omega^i \alpha\};\ \  i=1,2.
$$

%Moreover,
%$$
%\Lambda_X\otimes_{\Z[\omega]}\C\simeq H^3(X, \ \C)_{\omega}=H^{3,0}(X, \ \C)\oplus H^{2,1}(X, \ \C)
%$$
%$$
%H^3(X, \ \C)_{\bar{\omega}}=H^{1,2}(X, \ \C)\oplus H^{0,3}(X, \ \C).
%$$
By results in Section \ref{sec:Families from hyperplane arrangements}, we know that there exist six distinct points $p_1,\cdots,p_6$ on $\P^1$, such that $H_i$ can be identified with $\{p_i\}\times \P^1\times \cdots \times \P^1$. Moreover, let $C$ be the cyclic triple cover of $\P^1$ branched along $p_1,\cdots,p_6$, then the correspondence between $C$ and $X$ shows that the lattice $(H^3(X, \ \Z[\omega])_{\omega}, h)$ and $(H^1(C, \ \Z[\omega])_{\omega}, h)$ are isomorphic, both of which are rank four $\Z[\omega]$-lattice with signature $(3,1)$. Here the Hermitian form $h$ on $H^3(X, \ \Z[\omega])$ is defined by $h(\alpha,\beta)=-i\ Q( \alpha, \bar{\beta})$, in the same way as the curve case.

Let $(\Lambda, h)$ be a  $\Z[\omega]$-lattice of signature $(3,1)$, and let $\rho: \pi_1(\mx, s)\rightarrow Aut(\Lambda, h)$ be the monodromy representation of the family $f_3$.  A family version of the correspondence in Section \ref{sec:Families from hyperplane arrangements} shows that under  the association isomorphism $\phi: \mc\rightarrow \mx$,   if $s_2=\phi(s_1)$, then
$$
\mathcal{X}_{AR,s_2}\simeq \frac{\mathcal{C}_{s_1}\times \mathcal{C}_{s_1} \times \mathcal{C}_{s_1}}{N\rtimes S_3}. $$

The isometry $\phi_{\Omega}: (\Lambda_{\mathcal{X}_{s_2}}, h)\xrightarrow{\sim}(\bar{\Lambda}_{\mathcal{C}_{s_1}}, h)$ implies the following commutative diagram:
\begin{equation}\label{diagram:monodromy goups}
\begin{diagram}
\pi_1(\mc,s)&  &\rTo^{\sim}_{\phi_*}& &\pi_1(\mx,\phi(s))\\
&\rdTo^{P\rho}& & \ldTo^{P\rho}\\
& &PAut(\Lambda, h)
\end{diagram}
\end{equation}

Keeping the same notations as in Section \ref{sec:monodromy group and period map:curve case}, the commutative diagram (\ref{diagram:monodromy goups}) and  Corollary (\ref{cor:projectified monodromy group}) give the  following proposition.

\begin{proposition}\label{prop:projectified monodromy group:CY case}
The projectified monodromy representation
$$
P\rho: \pi_1(\mx, s)\rightarrow PAut(\Lambda, h)
$$
has image $P\Gamma_{\theta}$.    \proofend
\end{proposition}

Moreover, by Proposition (\ref{prop:summarize Hodge structures on X, C}), we have the decompositions
\begin{equation}\label{equation:decomposition}
\begin{split}
 H^3(X, \ \C)_{\omega}=H^{3,0}(X, \ \C)\oplus H^{2,1}(X, \ \C)\\
 H^3(X, \ \C)_{\bar{\omega}}=H^{3,0}(X, \ \C)\oplus H^{2,1}(X, \ \C).
\end{split}
\end{equation}
By associating the isomorphic class of $X$ with the point $[H^{3,0}(X, \ \C)]$ in $\B_3\simeq \{v\in \P(\Lambda\otimes_{\Z[\omega]}\C | h(v, v)<0)\}$, we get the period map $P_X: \mx^{uni}\rightarrow \B_3$, where $\mx^{uni}$ is the universal cover of $\mx$. As in the curve case, the period map $P_X$ is equivariant under the projectified monodromy representation $P\rho: \pi_1(\mx, s)\rightarrow    PAut(\Lambda, h)$.
Let $K$ be the kernel of $P\rho$, then $P_X$ descends to the   holomorphic map  $\tmx:=\mx^{uni}/K\rightarrow \B_3$, still denoted by $P_X$.

The association isomorphism $\phi: \mc\xrightarrow{\sim} \mx$ gives the commutative diagram relating period maps:
\begin{diagram}
\mc^{uni}&  &\rTo^{\sim}_{\phi}& & \mx^{uni}\\
&\rdTo^{P_C}& & \ldTo^{P_X}\\
& & \B_3
\end{diagram}

This diagram descends to the following commutative diagram:
\begin{equation}\label{equ:diagram for smx and smc}
\begin{diagram}
\tmc&  &\rTo^{\sim}_{\phi}& & \tmx\\
&\rdTo^{P_C}& & \ldTo^{P_X}\\
& & \B_3
\end{diagram}
\end{equation}

Recall $\smc$ is the moduli space of stable six ordered points on $\P^1$.
Let $\smx$ be the moduli space of stable six ordered hyperplanes in $\P^3$, which consists six ordered hyperplanes in $\P^3$  with at worst four-fold intersection point.
Then the association isomorphism $ \mc\xrightarrow{\sim} \mx$ extends to an isomorphism $\smc\xrightarrow{\sim} \smx$, and further extends to an isomorphism between Fox completions:
\begin{diagram}
\tsmc&  \rTo^{\sim}_{\phi_{ass}} & \tsmx\\
\dTo &     & \dTo\\
\smc& \rTo^{\sim} & \smx
\end{diagram}

We have the following proposition.
\begin{proposition}\label{prop:commutative dia for period mappings}
There exists a unique isomorphism  $\tsmx\xrightarrow{P_X} \B_3$ extending  the period map $\tmx\xrightarrow{P_X} \B_3$. Moreover, the following diagram is commutative:
\begin{diagram}
\tsmc&  \rTo^{P_C}_{\sim} & \B_3\\
\dTo^{\phi_{ass}}_{\wr} & \ruTo^{P_X}_{\sim}    & \\
\tsmx& &
\end{diagram}
\end{proposition}
\begin{proof} By \cite{DM}, the period map $P_C:\mc\rightarrow \B_3$ extends to an isomorphism $\tsmc\xrightarrow{\sim} B_3$.
Since $\tsmx \setminus \smx$ is a normal crossing divisor, the extendability and uniqueness follow from Riemann's extension theorem. The commutative diagram follows from the diagram (\ref{equ:diagram for smx and smc}).
\end{proof}

Similarly as Proposition \ref{prop:period iso: curve case}, we have the following proposition.
\begin{proposition}\label{prop:period iso: CY case}
The period mapping induces a bijective map: $\smx/S_6 \xrightarrow{\sim} \B_3/Aut(\Lambda, h)$. \proofend
\end{proposition}

As a corollary, we have the following global Torelli type theorem.
\begin{theorem}\label{thm:global Torelli}
 Suppose $a=(H_1,\cdots, H_6)$ and $b=(H^{'}_1,\cdots,H^{'}_6)$ are two hyperplane arrangements in general position in $\P^3$. Let $X_a$(resp. $X_b$) be the cyclic triple cover of $\P^3$ branched along $a$ (resp. $b$). Then the polarized $\Z$-Hodge structures $H^3(X_a, \ \Z)$ and $H^3(X_b, \ \Z)$ are isomorphic if and only if after a permutation, the hyperplane arrangements $\{H_1,\cdots, H_6\}$ and $\{H^{'}_1,\cdots,H^{'}_6\}$ are projectively equivalent.
\end{theorem}
\begin{proof}
Let $\phi:H^3(X_a, \ \Z)\xrightarrow{\sim}H^3(X_b, \ \Z)$ be an isomorphism of polarized $\Z$-Hodge structures. By the decomposition (\ref{equation:decomposition}), we see that $\phi$ is compatible with the $\Z/3\Z$-actions. Then $\phi$ induces an isomorphism $H^3(X_a, \ \Z[\omega])_{\omega}\xrightarrow{\sim}H^3(X_b, \ \Z[\omega])_{\omega}$.  From this we see $a$ and $b$ have the same image under the period map $\smx \xrightarrow{} \B_3/Aut(\Lambda, h)$. Then Proposition \ref{prop:period iso: CY case} implies $a$ and $b$ represent the same point in $\smx/S_6$, which means exactly that after a permutation, the hyperplane arrangements $a$ and $b$ are projectively equivalent.
\end{proof}

\section{Analysis of stable degenerations }
In this section, we want to extend the global Torelli type Theorem \ref{thm:global Torelli} to the stable hyperplane arrangement case.

We first analyze boundary correspondence under the period mapping.
Recall $(\Lambda, h)$ is a rank four $\Z[\omega]$ lattice with signature $(3,1)$, and we realize $\B_3$  as the open subset of $\P(\Lambda\otimes_{\Z[\omega]}\C)$ consisting of negative lines. We call a vector $r\in \Lambda$  a short root, if $h(r,r)=1$.  Denote $R$ for  the set of short roots in $\Lambda$. For any $r\in R$, define the hyperplane orthogonal to $r$:
$$
H_r:=\{[v]\in \B_3\subset \P(\Lambda\otimes_{\Z[\omega]}\C) | v\in \Lambda, \ h(v,r)=0 \}.
$$

We write $\mathcal{H}:=\cup_{r\in R}H_r$. Proposition \ref{prop:commutative dia for period mappings} gives an isomorphism $P_X:\tsmx\xrightarrow{\sim} \B_3$.

\begin{proposition}
$P_X$ induces  isomorphisms $\tmx\xrightarrow{\sim}\B_3\setminus \mathcal{H}$ and $\tsmx\setminus \tmx \xrightarrow{\sim} \mathcal{H}$.
\end{proposition}
\begin{proof}
Note first that  $\tsmx\setminus \tmx $ (resp. $ \mathcal{H}$ ) is a union of $15$ irreducible hypersurfaces in the complex manifold $\tsmx$ (resp. $\B_3$). Since $P_X:\tsmx\xrightarrow{\sim} \B_3$ is an isomorphism, it suffices to show $P_X(\tmx)\subset \B_3\setminus \mathcal{H}$. If $x\in \tmx$ and $P_X(x)\in H_r$ for some $r\in R$,  then we can choose an $\omega$-reflection $\alpha_r$ along $r$ in $P\Gamma_{\theta}$. In particular, $r$ is a fixed point of $\alpha_r$. Next we  choose $\gamma_r\in \pi_1(\tmx, s)$ satisfying $P\rho(\gamma_r)=\alpha_r$. Since the period mapping $\tmx\xrightarrow{P_X} \B_3$ is $\pi_1(\tmx, s)$-equivariant, we see $P_X(\gamma_r(x))=\alpha_r(P_X(x))=P_X(x)$. Since $P_X$ is injective, we see $\gamma_r(x)=x$. Note the cover $\tmx\rightarrow \mx$ is Galois, so any nontrivial  deck transformation has no fixed points. This implies $\gamma_r$ belongs to the kernel of the monodromy representation $P\rho: \pi_1(\tmx, s)\rightarrow P\Gamma_{\theta}$, and hence $\alpha_r=P\rho(\gamma_r)=id$ is the identity element in $P\Gamma_{\theta}$. This contradicts with the chosen of $\alpha_r$. So we get   $P_X(\tsmx\setminus \tmx)\subset \mathcal{H}$.
\end{proof}

In order to  give a geometric interpretation of the period mapping $P_X$ on the boundary $\tsmx\setminus \tmx$, we study the Hodge structure on cyclic triple covers of $\P^3$ branched along stable six hyperplanes.
For $a=(H_1,\cdots, H_6)\in \smx\backslash \mx$, we can  still define $X_a$ to be the cyclic triple cover of $\P^3$ branched along the divisor $\sum_{i=1}^6H_i$.
\begin{proposition}\label{prop:stable CY has pure Hodge structure}
For each $a\in \smx$, Deligne's mixed  Hodge structure on $H^3(X_a, \Q)$ is pure.
\end{proposition}
\begin{proof}
Denote the homogeneous coordinates on $\P^3$ by $[X_0:\cdots: X_3]$, and for $1\leq i\leq 6$, let $\ell_i$ be the defining homogeneous linear equation of the hyperplane $H_i$. Over  the rational function field $K(\P^3)$ of $\P^3$,  there exists the finite Galois extension $L:=K(\P^3)(\sqrt[3]{\frac{\ell_2}{\ell_1}}, \cdots, \sqrt[3]{\frac{\ell_6}{\ell_1}})$. Define  $Y_a$ to be  the normalization of $\P^3$ in the Galois extension field $L$. It is not hard  to see that $Y_a$ is a complete intersection of two degree three hypersurfaces in $\P^5$, and $Y_a$ is smooth if $a\in \mx$. Moreover, the finite abelian group $N_1$, which is defined as the kernel of the summation homomorphism $\sum_{i=0}^5\Z/3\Z\rightarrow \Z/3\Z $, acts on $Y_a$, and $X_a$ is isomorphic to the quotient variety $Y_a/N_1$. For details of these claims, one can see \cite{SXZ2}, section 2.2.

Since $X_a\simeq Y_a/N_1$, we obtain that the  (mixed) Hodge structure  $H^3(X_a, \ \Q)$ is isomorphic to $H^3(Y_a, \ \Q)^{N_1}$, the $N_1$-invariant part of $H^3(Y_a, \ \Q)$.

If $a\in \mx$, then the mixed  Hodge structure on $H^3(Y_a, \ \Q)$ is pure, since in this case $Y_a$ is a smooth projective variety. So the mixed Hodge structure on $H^3(X_a, \ \Q)\simeq H^3(Y_a, \ \Q)^{N_1}$ is also pure.

If $a=(H_1,\cdots, H_6)\in \smx\backslash \mx$, we can see that $Y_a$ has only isolated singularities. In order to show the mixed Hodge structure on $H^3(X_a, \ \Q)\simeq H^3(Y_a, \ \Q)^{N_1}$ is pure, we can assume, without loss of generality, that the first four hyperplanes $H_1$, $H_2$, $H_3$, $H_4$ pass through a common point $p=[1:0:0:0]\in \P^3$, and $\sum_{i=1}^6H_i$ is a normal crossing  divisor on $\P^3\backslash \{p\}$.  By an automorphism of $\P^3$, we can assume the defining equations of $H_i$ ($1\leq i\leq 6$) are the following:
\begin{equation}\notag
\ell_1=X_0;\
\ell_2=X_1;\
\ell_3=X_2;\
\ell_4=X_3;\
\ell_5=X_1+X_2+X_3;\
\ell_6=X_0+b_1X_1+b_2X_2+b_3X_3.
\end{equation}
Here $b_i$ ($1\leq i\leq 3$) are complex numbers. Then we can see $Y_a$ is the complete intersection in $\P^5$ defined by the following two homogeneous equations:
\begin{equation}\notag
\begin{split}
Y_4^3&=Y_1^3+Y_2^3+Y_3^3;\\
Y_5^3&=Y_0^3+b_1Y_1^3+b_2Y_2^3+b_3Y_3^3
\end{split}
\end{equation}
where $[Y_0:\cdots:Y_5]$ are the homogeneous coordinates on $\P^5$. The quotient morphism from $Y_a$ to $X_a$ is \begin{equation}\notag
\begin{split}
\pi:Y_a&\rightarrow X_a\\
[Y_0:\cdots:Y_5]&\mapsto [Y_0^3:Y_1^3:Y_2^3:Y_3^3].
\end{split}
\end{equation}
The finite abelian group $N_1=ker(\sum_{i=0}^5\Z/3\Z\xrightarrow{\sum}\Z/3\Z)$ acts on $Y_a$ by the following way:
\begin{equation}\notag
\begin{split}
N_1\times Y_a&\rightarrow Y_a\\
((a_0,\cdots, a_5),[Y_0:\cdots:Y_5])&\mapsto [\omega^{a_0}Y_0:\cdots:\omega^{a_5}Y_5]
\end{split}
\end{equation}
where $\omega$ is a primitive cubic root of unity.

It is easy to see the singular subset of $Y_a$ is $\pi^{-1}([1:0:0:0])=\{[1:0:0:0:0:\omega^i]| i=0,1,2\}$.
Blowing up $Y_a$ along these singular points, we get a smooth projective variety $\tilde{Y}_a$, and we can see the exceptional divisor $E$ on $\tilde{Y}_a$ is a disjoint union of smooth cubic surfaces.

Now we take a sufficient small open neighborhood $V$ of $[1:0:0:0]$ in $\P^3$, such that $V$ is biholomorphic to an open ball. Let $U_1$ be the inverse image of $V$ in $\tilde{Y}_a$, and let $U_2=\tilde{Y}_a\backslash E$. Then $\tilde{Y}_a=U_1\cup U_2$ and we get the following exact sequence from the Meyer-Vietoris sequence:
$$
H^2(U_1\cap U_2, \ \Q)^{N_1}\rightarrow H^3(\tilde{Y}_a, \ \Q)^{N_1}\rightarrow H^3(U_1, \ \Q)^{N_1}\oplus H^3(U_2, \ \Q)^{N_1}\rightarrow H^3(U_1\cap U_2, \ \Q)^{N_1}.
$$

It is not difficult to see that the exceptional divisor  $E$ is a deformation retract of $U_1$. Since $E$ is a disjoint union of smooth cubic surfaces, we get $H^3(U_1, \ \Q)\simeq H^3(E, \ \Q)=0$. On the other hand, we can see $U_2/N_1\simeq X_a\backslash \{p\}$, where $p$ is the inverse image of $[1:0:0:0]\in \P^3$. This implies the isomorphism  $H^3(U_2, \ \Q)^{N_1}\simeq H^{3}(X_a\backslash \{p\}, \ \Q)$.

Next we consider $H^2(U_1\cap U_2, \ \Q)^{N_1}$ and $H^3(U_1\cap U_2, \ \Q)^{N_1}$. Let $Z$ by the hypersurface in $\C^4$ defined by the equation $x_4^3=x_1x_2x_3(x_1+x_2+x_3)$, then we can see the quotient space $U_1\cap U_2/N_1$ is homotopic to $Z\backslash \{(0,0,0,0)\}$. Since a direct computation shows that both the cohomology groups $H^2(Z\backslash \{(0,0,0,0)\}, \ \Q)$ and $H^3(Z\backslash \{(0,0,0,0)\}, \ \Q)$ are vanishing, we see $H^2(U_1\cap U_2, \ \Q)^{N_1}=H^3(U_1\cap U_2, \ \Q)^{N_1}=0$. Moreover, a similar Meyer-Vietoris sequence analysis on $X_a$ shows that $H^3(X_a\backslash \{p\}, \ \Q)\simeq H^3(X_a, \ \Q)$.

Putting the above analysis together, we conclude $H^3(X_a, \ \Q)\simeq H^3(\tilde{Y}_a, \ \Q)^{N_1}$. Since $\tilde{Y}_a$ is a smooth projective variety, we see the mixed Hodge structure on $H^3(X_a, \ \Q)$  is pure.

\end{proof}

For a stable hyperplane arrangement $a=(H_1,\cdots, H_6)\in \smx$, the divisor $\sum_{i=1}^6 H_i$ has at worst four-fold intersection point, and the number of four-fold intersection points is less than or equal to $3$. Furthermore, if the number of four-fold intersection points is $k$, then $dim_{\Q}H^3(X_a, \Q)=8-2k$, and the $(H^3(X_a, \Z[\omega])_{\Z[\omega]}, h)$ is a rank $4-k$ lattice with signature $(3-k,1)$. All these facts can be checked by a careful degeneration analysis as that in  the proof of Proposition \ref{prop:stable CY has pure Hodge structure}. Since it is routine and a little tedious, we omit it here.

Now we can extend Theorem \ref{thm:global Torelli} to  stable hyperplane arrangements.
\begin{theorem}\label{thm:global Torelli:stable case}
 Suppose $a=(H_1,\cdots, H_6)$ and $b=(H^{'}_1,\cdots,H^{'}_6)$ are two stable hyperplane arrangements in $\P^3$. Let $X_a$(resp. $X_b$) be the cyclic triple cover of $\P^3$ branched along $a$ (resp. $b$). Then the polarized $\Z$-Hodge structures $H^3(X_a, \ \Z)$ and $H^3(X_b, \ \Z)$ are isomorphic if and only if after a permutation, the hyperplane arrangements $\{H_1,\cdots, H_6\}$ and $\{H^{'}_1,\cdots,H^{'}_6\}$ are projectively equivalent.
\end{theorem}
\begin{proof}
Suppose $dim_{\Q}H^3(X_a, \Q)=dim_{\Q}H^3(X_b, \Q)=8-2k$, with $k=0,1,2,3$. Then we discuss the four cases according to $k$.
\begin{enumerate}
\item[(1)] $k=0$: In this case, both $a$ and $b$ are in general position, so it follows from Theorem \ref{thm:global Torelli}.
\item[(2)] $k=1$:  Since $\smx\backslash \mx \simeq \mathcal{H}/Aut(\Lambda, h)$ is irreducible, we can choose an irreducible component $H$ of $\tsmx\setminus \tmx$ and $\tilde{a}$, $\tilde{b}$ in $H$ over $a$ and $b$ respectively. Suppose the isomorphism  $P_X: \tsmx\xrightarrow{\sim} \B_3$ maps $H$ to  $\B_2^r:=\{[v]\in \P(\Lambda^r\otimes_{\Z[\omega]}\C)| h(v,v)<0\}$. Here $r\in R$ is a short root, and  $\Lambda^r:= \{\alpha\in \Lambda| h(\alpha, r)=0\}$ is a free $\Z[\omega]$-module of rank three.  we can see the isomorphism $P_X: H\xrightarrow{\sim} \B_2^r$ coincides with the period map $H\rightarrow \B_2^r$ which associates with a point  $c$ in the smooth locus of $H$ the line $[H^{3,0}(X_c, \ \C)]$ in $\P(H^3(X_c, \Z[\omega])_{\omega}\otimes_{\Z[\omega]}\C)$, where $X_c$ is the cyclic triple cover of $\P^3$ branched along $c$. Then since the polarized $\Z$-Hodge structures $H^3(X_a, \ \Z)$ and $H^3(X_b, \ \Z)$ are isomorphic, we get $P_X(\tilde{a})=P_X(\tilde{b})$. By Proposition \ref{prop:period iso: CY case}, this in turn implies $a$ and $b$ has the same image in $\smx/S_6$, so after a permutation, the hyperplane arrangements $a$ and $b$ are projectively equivalent.
\item[(3)] $k=2$: In this case, we can choose two irreducible components $H_1,H_2$ of $\tsmx\setminus \tmx$ and points $\tilde{a}$, $\tilde{b}$ in $H_1\cap H_2$ over $a$ and $b$ respectively. Moreover, if the isomorphism  $P_X: \tsmx\xrightarrow{\sim} \B_3$ maps $H_i$ to $\B_2^{r_i}$, $i=1,2$, then $r_1$ and $r_2$ are two perpendicular short roots. Similar as the $k=1$ case, the restriction $P_X: H_1\cap H_2\xrightarrow{\sim} \B_2^{r_1}\cap \B_2^{r_2}$ coincides with the period map by associating a point  $c$ in the smooth locus of $H_1\cap H_2$ the line $[H^{3,0}(X_c, \ \C)]$ in $\P(H^3(X_c, \Z[\omega])_{\omega}\otimes_{\Z[\omega]}\C)$. Then this case follows from the same argument as in the $k=1$ case.
\item[(4)] $k=3$: In this case, both $a$ and $b$ have three fourfold intersection points, then an elementary analysis on the configuration of hyperplane arrangements shows that after a permutation, the hyperplane arrangements $a$ and $b$ are projectively equivalent.
\end{enumerate}
\end{proof}

%\appendix{Comparison with Rohde's example}\label{sec:comparison with Rohde's example}
%\section{Comparison with Rohde's example}\label{sec:comparison with Rohde's example}
\appendix%{Comparison with Rohde's example}\label{sec:comparison with Rohde's example}
  \renewcommand{\appendixname}{Appendix ~\Alph{section}}
 \section{Comparison with Rohde's example}\label{sec:comparison with Rohde's example}

Given distinct $a,b,c\in \C\backslash \{0,1\}$, Rohde \cite{Rohde} constructed a singular Calabi-Yau threefold $X^{'}$ in the following way:

Let $W$ be the surface in the weighted projective space $\P(2,2,1,1)$ defined by the equation $y_1^3+y_2^3+x_0x_1(x_1-x_0)(x_1-ax_0)(x_1-bx_0)(x_1-cx_0)=0$. Let $F$ be the Fermat curve in $\P^2$ defined by the homogeneous equation $z_0^3+z_1^3+z_2^2=0$. Then the cyclic group $G=\Z/3\Z$ acts on $W$ and $F$. Fixing a generator $\sigma$ of $G$ and let $\omega$ be a fixed primitive cubic root of unity. We define these actions explicitly:
\begin{equation}\notag
\begin{split}
\sigma: W&\rightarrow W \\
 [x_0:x_1:y_1:y_2]&\mapsto [x_0:x_1:\omega y_1:\omega y_2]\\
 \sigma: &F\rightarrow F\\
 [z_0:z_1:z_2]&\mapsto [\omega z_0:z_1:z_2]
\end{split}
\end{equation}

Rohde \cite{Rohde} constructs the Calabi-Yau threefold as a crepant resolution of  the quotient threefold $X^{'}:=W\times F/G$, where $G$  acts on $W\times F$ in the diagonal way. Moreover, varying the parameters $a, b,c$ in $\C\backslash \{0,1\}$, Rohde obtains a family of Calabi-Yau threefolds $\mathcal{X}^{'}\rightarrow \mc $, where we recall that $\mc $ is the moduli space of ordered six distinct points in $\P^1$. The main goal of this section is to show Rohde's family is birationally equivalent to the  family $\mathcal{X}_{AR}\rightarrow \mathfrak{M}_{3, 6}$, which  is the universal family of cyclic triple covers of $\P^3$ branched along six hyperplanes in general position.

We first analyze the structure of the singular surface $W$. In general, if $X_1$ and $X_2$ are two varieties with $G$-action, we say $X_1$ and $X_2$ are $G$-birationally equivalent  if there exists a birational map from $X_1$ to $X_2$ compatible with the $G$-action.

Let $C$ be the cyclic triple cover of $\P^1$ branched along the six points $\{0, 1, \infty, a, b, c\}$ whose affine model is the curve in $\C^2$ defined by the equation $y^3-x(x-1)(x-a)(x-b)(x-c)=0$. Then $G$ acts on $C$ in the following way:
\begin{equation}\notag
\begin{split}
\sigma: C &\rightarrow C\\
(x,y)& \mapsto  (x, \omega y)
\end{split}
\end{equation}

Let $G$ act on the product $C\times F$ diagonally, then $G$ acts on the quotient $C\times F/ G$ through the $G$-action on the  first factor $C$.
\begin{lemma}\label{lemma: W birational to CtimesF/G}
$W$ is $G$-birationally  equivalent to the quotient $C\times F/ G$.
\end{lemma}
\begin{proof}
Let $F_0$ be the affine surface $\{(z_1,z_2)\in \C^2| 1+z_1^3+z_2^3=0\}$, and define a $G$-action on $F_0$ by $\sigma (z_1,z_2)=(\omega z_1,\omega z_2)$. Then $F_0$ is $G$-birational to the Fermat curve $F$. Define the following morphism:
\begin{equation}\notag
\begin{split}
C\times F_0&\rightarrow W \\
(x,y,z_1,z_2)&\mapsto (x, z_1 y, z_2 y)
\end{split}
\end{equation}
It is easy to see this morphism induces a $G$-birational equivalent between $C\times F_0/ G$ and $W$. So $W$ is $G$-birationally  equivalent to $C\times F/ G$.
\end{proof}

Now  we consider the following six hyperplanes in $\P^3$ which are in general position
$$
H_i: X_i=0 \ (0\leq i\leq 3), \ \ H_4: \sum_{i=0}^3X_i=0, \ \ H_5: X_0+a X_1+b X_2+c X_3=0,
$$
where $[X_0:\cdots:X_3]$ is the homogeneous coordinates on $\P^3$. Let $X$ be the cyclic triple cover of $\P^3$ branched along $\sum_{i=0}^5 H_i$.

In order to analyze the structure of $X$, we define some auxiliary  varieties.

Let $u_1$, $v_1$ be linear functions of $u$, $v$ defined by the following relations:
\begin{equation}\label{equ:u,v and u_1, v_1}
\left\{
               \begin{array}{ll}
                u_1 & =1+u+v \\
                v_1 & =a+bu+cv.

 \end{array}
             \right.
\end{equation}

We define $Y$ as the following affine variety:
$$
Y=\{(t_1,u,v,y_1)\in \C^4| y_1^3=\frac{uvt_1(t_1+1)}{u_1v_1(v_1-u_1)}\}.
$$

Let $S$ be the following affine surface:
 $$
 S=\{(w, u, v)\in \C^3| w^3=\frac{uv}{u_1v_1(v_1-u_1)}\}.
 $$
Let $G$ acts on $S$ by $\sigma(w, u, v)=(\omega^2 w, u, v)$.

\begin{lemma}\label{lemma:X birational to various objects}
We have the following birational isomorphsims:
\begin{itemize}
\item [(1)] $X$ is birationally equivalent to $Y$.
\item [(2)] $Y$ is birationally equivalent to $S\times F /G$, where $G$ acts on $S\times F$ diagonally.
\item [(3)] $S$ is $G$-birationally equivalent to $C\times F /G$.
\end{itemize}
\end{lemma}
\begin{proof}
(1) We take the following affine model of $X$:
$$
X_1=\{(x_1,x_2,x_3,y)\in \C^4| y^3=x_1x_2x_3(1+x_1+x_2+x_3)(1+ax_1+bx_2+cx_3)\}.
$$
Under the coordinate transformation
$$
\left\{
               \begin{array}{ll}
                x_1 & =t \\
                x_2 & =tu\\
                x_3 & =tv\\
                y   & =y
 \end{array}
             \right.
$$
the hypersurface $X_1$ is birationally equivalent to the following hypersurface in $\C^4$:
$$
X_2=\{(t,u,v,y)\in \C^4| y^3=t^3uv(1+u_1t)(1+v_1t)\},
$$
where $u_1$, $v_1$ are defined by the equations (\ref{equ:u,v and u_1, v_1}).
Then we can see $X_2$ is birational to $Y$ under the following coordinate transformation:
$$
\left\{
               \begin{array}{ll}
                u & =u \\
                v & =v\\
                t & =(\frac{1}{v_1}-\frac{1}{u_1})t_1-\frac{1}{u_1}\\
                y   & =tu_1v_1(\frac{1}{u_1}-\frac{1}{v_1})y_1.
 \end{array}
             \right.
$$

(2) It is direct to see the following affine curve $F_1$ is $G$-birationally equivalent to the Fermat curve $F$:
$$
F_1=\{(t_1,x)\in \C^2| x^3=t_1(t_1+1)\},
$$
where $G$ acts on $F_1$ by $\sigma (t_1, x)=(t_1, \omega x)$.

Then the following rational map induces the desired $G$-birationally equivalence between $S\times F /G$ and $Y$:
\begin{equation}\notag
\begin{split}
S\times F_1&\rightarrow Y \\
(w, u, v, t_1, x)&\mapsto (t_1, u, v, wx).
\end{split}
\end{equation}

(3) From the equations (\ref{equ:u,v and u_1, v_1}), we can view $(w, u_1, v_1)$ as coordinate system on $\C^3$, and we make the following coordinate transformation:
$$
\left\{
               \begin{array}{ll}
                w & =w \\
                u_1 & =t_2\\
                v_1 & =t_2z.
 \end{array}
             \right.
$$
Under this  coordinate transformation, we see $S$ is $G$-birationally equivalent to the following surface:
$$
S_1=\{(w,t_2,z)\in \C^3| w^3=\frac{(A_1t_2+a_1)(B_1t_2+b_1)}{t_2^3z(z-1)}\},
$$
where
$$
\left\{
               \begin{array}{ll}
               A_1 & =\frac{b-z}{b-c} \\
                B_1 & =\frac{c-z}{c-b}\\
                a_1 & =\frac{a-b}{b-c}\\
                b_1 & =\frac{a-c}{c-b},
 \end{array}
             \right.
$$
and $G$ acts on $S_1$ by $\sigma (w, t_2, z)=(\omega^2 w, t_2, z)$.

A further coordinate transformation:
$$
\left\{
               \begin{array}{ll}
               z&= z\\
                t_3 & =\frac{t_2+\frac{a_1}{A_1}}{\frac{b_1}{B_1}-\frac{a_1}{A_1}} \\
               w_1 & =\frac{wt_2}{A_1B_1(\frac{b_1}{B_1}-\frac{a_1}{A_1})}
               \end{array}
             \right.
$$
shows  that $S_1$ is $G$-birational to the  surface:
$$
S_2=\{(z,t_3,w_1)\in \C^3|w_1^3=\frac{t_3(t_3+1)}{z(z-1)A_1B_1(A_1b_1-B_1a_1)} \},
$$
where $G$ acts on $S_2$ by $\sigma (z, t_3, w_1)=(z, t_3, \omega^2 w_1)$.

Now let $C_1$ be  the following affine curve:
$$
C_1:=\{(z,w_2)\in\C^2| w_2^3=z(z-1)A_1B_1(A_1b_1-B_1a_1)\},
$$
and let $G$ act on $C_1$ by $\sigma (z, w_2)=(z, \omega^2 w_2)$.
We consider the affine model
$$
F_1=\{(t_3,x)\in \C^2| x^3=t_3(t_3+1)\}
$$
of the Fermat curve $F$ as before.
The rational map
\begin{equation}\notag
\begin{split}
C_1\times F_1&\rightarrow S_2 \\
(z, w_2, t_3, x)&\mapsto (z, t_3, \frac{x}{w_2}).
\end{split}
\end{equation}
gives a $G$-birationally equivalence between $C_1\times F_1 /G$ and $S_2$.
Moreover, we see the smooth projective model of $C_1$ is  isomorphic to  $C$,  the cyclic triple cover of $\P^1$ branched along the   six points
$\{0,1,\infty, a, b, c\}$.  By combining all of the birational equivalences above,  we  obtain $S$ is $G$-birationally equivalent to $C\times F/ G$.
\end{proof}

Now it is direct to see that, by combining Lemma \ref{lemma: W birational to CtimesF/G} and Lemma \ref{lemma:X birational to various objects}, we obtain the following birational equivalence.
\begin{proposition}\label{prop:Rohde's CY birational to our CY}
Given distinct  $a,b,c\in \C\backslash \{0,1\}$,
Rohde's singular Calabi-Yau threefold $X^{'}=W\times F /G$ is birational to $X$, which is the cyclic triple cover of $\P^3$ branched along $\sum_{i=0}^5 H_i$, with $H_i$ defined by
$$
H_i: X_i=0 \ (0\leq i\leq 3), \ \ H_4: \sum_{i=0}^3X_i=0, \ \ H_5: X_0+a X_1+b X_2+c X_3=0.
$$
\end{proposition}

Note that to give  distinct $a,b,c\in \C\backslash \{0,1\}$ is   equivalent to give six distinct points $\{0, 1, \infty, a, b, c\}$ in $\P^1$. Since the moduli space $\mc $ of ordered six distinct points in $\P^1$ is isomorphic to the moduli space  $\mathfrak{M}_{3, 6}$ of ordered six hyperplane arrangements in general in $\P^3$, the following birational equivalence is a direct consequence of Proposition \ref{prop:Rohde's CY birational to our CY}.
\begin{proposition}\label{prop:family version:Rohde's CY birational to our CY}
Rohde's Calabi-Yau  family $\mathcal{X}^{'}\rightarrow \mc $ is birationally equivalent to the
 universal family $\mathcal{X}_{AR}\rightarrow \mathfrak{M}_{3, 6}$ of cyclic triple covers of $\P^3$ branched along six hyperplanes in general position.
\end{proposition}

\iffalse
\textbf{Acknowledgements}  We would like to thank Professor Chin-Lung Wang for his interest to this work.
\fi

\end{document}